\documentclass[11pt,a4paper]{article}
\usepackage[T1]{fontenc}
\usepackage{newpxtext}
\usepackage{amsmath,amsthm,mathtools}
\usepackage{newpxmath}
\usepackage[margin=27mm,headheight=15pt,headsep=8mm,footskip=12mm]{geometry}
\usepackage{microtype,booktabs,array,needspace,fancyhdr}
\usepackage{xcolor}
\definecolor{linkink}{RGB}{25,67,86}
\usepackage[colorlinks=true,linkcolor=linkink,citecolor=linkink,urlcolor=linkink,
  bookmarksnumbered=true,pdfencoding=auto,psdextra]{hyperref}
\hypersetup{
  pdftitle={Sudakov minoration for weighted Orlicz balls: a readable proof},
  pdfauthor={Witold Bednorz, Rafal Martynek and Rafal Meller},
  pdfsubject={Historical introduction, monotone budget coupling, product comparison and residual absorption}}
\fancypagestyle{plain}{\fancyhf{}\fancyfoot[C]{\small\thepage}}
\numberwithin{equation}{section}
\allowdisplaybreaks[1]
\newtheorem{theorem}{Theorem}[section]
\newtheorem{lemma}[theorem]{Lemma}
\newtheorem{proposition}[theorem]{Proposition}
\newtheorem{corollary}[theorem]{Corollary}
\theoremstyle{definition}
\newtheorem{definition}[theorem]{Definition}
\theoremstyle{remark}
\newtheorem{remark}[theorem]{Remark}
\newcommand{\E}{\mathbb E}
\newcommand{\PP}{\mathbb P}
\newcommand{\R}{\mathbb R}
\newcommand{\ind}{\mathbf 1}
\newcommand{\cx}{\preceq_{\mathrm{cx}}}
\newcommand{\st}{\preceq_{\mathrm{st}}}
\newcommand{\Law}{\mathcal L}
\newcommand{\clip}{\operatorname{clip}}
\newcommand{\ip}[2]{\langle #1,#2\rangle}
\newcommand{\norm}[1]{\lVert #1\rVert}
\newcommand{\SMP}{\operatorname{SMP}}
\newcommand{\eps}{\varepsilon}
\newenvironment{keypoint}{\par\smallskip\begin{quote}\small\noindent\textbf{Key point.}\ }{\end{quote}\smallskip}
\newcommand{\proofstep}[1]{\par\smallskip\noindent\textbf{#1}\ }

\title{\textbf{Sudakov minoration for\ weighted Orlicz balls}\\[7pt]
  \large A readable proof with historical background\\[4pt]
  \normalsize Monotone budget coupling and local residual absorption}
\author{Witold Bednorz, Rafa\l{} Martynek and Rafa\l{} Meller\thanks{Research partially supported by Grant UMO-2022/47/B/ST1/02114.}\\[4pt]
  \small Institute of Mathematics, University of Warsaw\\[-2pt]
  \small Banacha 2, 02-097 Warszawa, Poland}
\date{}

\begin{document}
\maketitle
\begin{center}
\small\textbf{AI was used in this research.}
\end{center}
\begin{abstract}
We explain a dimension-free Sudakov minoration principle for the probability
measures with densities proportional to
\[
 e^{-\lambda\sum_i|x_i|}
 \ind_{\{\sum_i\phi_i(|x_i|)\le E\}},
\]
where the functions $\phi_i$ are finite, convex and strictly increasing.
The central idea is to compare this dependent distribution with a suitably
chosen product distribution. An auxiliary exponential variable makes the
product distribution an exact mixture over energy budgets. We couple these
budgets monotonically, correct the product vector into the desired
constraint, and show that the correction is small in every expected
seminorm. A comparison in the opposite direction preserves the moment
separation needed for Sudakov minoration.

The exposition includes the historical setting, a proof roadmap, detailed
coupling and calibration arguments, and explicit examples. No isotropic
normalization or uniform upper-growth assumption on the $\phi_i$ is needed.
The result concerns a single additive constraint; it does not assert the
same conclusion for all unconditional log-concave vectors.
\end{abstract}
\vspace{3pt}
\noindent\small\textbf{Mathematics Subject Classification:} 60G15, 60G17.\\
\textbf{Keywords:} Sudakov minoration; generalized Orlicz balls; log-concave
measures; stochastic order; convex order; coupling.\normalsize

\clearpage
\tableofcontents
\vspace{9pt}
\noindent\textbf{How to read the argument.}
Section~\ref{sec:history} gives the historical context.
Sections~\ref{sec:setting}--\ref{sec:transfer} state the result and explain
exactly which comparisons are sufficient. The construction itself is in
Sections~\ref{sec:budget}--\ref{sec:reverse};
Section~\ref{sec:completion} completes the proof.
Sections~\ref{sec:examples}--\ref{sec:scope} provide examples and clarify
what the estimates do, and do not, say. The independent-coordinate
minoration theorem is used as an external input, not reproved here.

\medskip
\noindent\textbf{The main comparison, in one line.}
The whole construction produces one product vector $V$, independent of
the index set under consideration, such that
\[
 \boxed{\quad X\cx 2MV,\qquad
       \E q(V)\le 2\E q(X)\quad\text{for every seminorm }q.\quad}
\]
The first inequality transfers separation to $V$; the second transfers a
lower bound for the expected supremum back to $X$.

\clearpage
\section{Historical introduction and perspective}\label{sec:history}

\subsection{From Gaussian separation to moment separation}
The original Sudakov minoration belongs to the geometric theory of Gaussian
processes developed in the late 1960s and early 1970s. Sudakov's work on
Gaussian measures and entropy \cite{Sudakov69}, followed by his geometric
study of Gaussian processes \cite{Sudakov}, provides the early background.
Its message is that a Gaussian process
cannot remain uniformly small on a large collection of well-separated
points. For a finite centered Gaussian family $(G_t)_{t\in T}$, put
\[
 d_G(s,t)=\bigl(\E|G_s-G_t|^2\bigr)^{1/2}.
\]
If $d_G(s,t)\ge r$ whenever $s\ne t$, then
\begin{equation}\label{eq:gaussian}
 \E\max_{t\in T}G_t\ge c\,r\sqrt{\log|T|}.
\end{equation}
The factor $\sqrt{\log|T|}$ is the gain obtained from having many separated
choices rather than just one.

To pass beyond Gaussian processes, it is useful to express separation in
terms of moments. If $g$ is a standard normal variable, then
$\norm{g}_p$ is comparable to $\sqrt p$ for $p\ge1$. Thus
\[
 \norm{G_t-G_s}_p\asymp \sqrt p\,d_G(s,t).
\]
Consequently, if $|T|\ge e^p$ and every increment has $L_p$ norm at least $a$,
\eqref{eq:gaussian} gives $\E\max_tG_t\ge c'a$.
This is the form that survives for non-Gaussian canonical processes.

The scale $p\simeq\log|T|$ is also natural from the elementary upper bound
\[
 \E\max_{t\in T}|Z_t|
 \le\Bigl(\sum_{t\in T}\E|Z_t|^p\Bigr)^{1/p}
 \le e\max_{t\in T}\norm{Z_t}_p
 \qquad (|T|\le e^p).
\]
Sudakov minoration asks for a converse under \emph{increment separation}.
It is not a consequence of this upper bound: the dependence among the
variables is precisely what makes a converse difficult.

\subsection{Independent coordinates and Bernoulli processes}
A canonical process has the form
\[
 Z_t=\sum_{i=1}^d t_iZ_i,\qquad t\in T\subset\R^d.
\]
An important part of the broader development was Talagrand's work on
infinitely divisible processes \cite{Talagrand93}: it extended
majorizing-measure lower bounds beyond the Gaussian setting and emphasized
the role of positive processes. For the independent-coordinate problem,
Lata\l a's 1997 paper \cite{Latala97} proved a universal Sudakov minoration
for independent symmetric coordinates with log-concave tails. The later
work of Lata\l a and Tkocz \cite{LT} treated independent coordinates with
regularly growing moments. These results supply the independent reference
case needed below. Their hypotheses do not require identical marginal
distributions.

The Bernoulli theory forms another part of this background. For independent
signs, a single Gaussian metric does not describe the entire supremum
problem. Bednorz and Lata\l a's characterization of bounded Bernoulli
processes \cite{BL} identifies a decomposition into a Gaussian contribution
and a contribution controlled by an $\ell_1$ bound. The present argument
uses only sign contraction and independent-coordinate minoration; it does
not use that characterization as a black box.

\subsection{The dependent log-concave setting}
Lata\l a \cite{Latala14} formulated a general Sudakov-type conjecture for
log-concave vectors and proved several special cases. In particular,
Theorem~5.2 and Corollary~5.3 of that paper give a radial transfer principle
and cover uniform measures on ordinary $\ell_r$ balls. Bednorz
\cite{Bednorz14} developed related reductions for dependent unconditional
vectors, including methods based on common witnesses.

The present family is naturally connected with that program. The
coordinates need not be independent: the constraint
$\sum_i\phi_i(|x_i|)\le E$ forces them to share a common energy budget.
Furthermore, for general choices of the $\phi_i$, the sublevel sets are not
scalar dilates of one fixed convex body. A radial representation therefore
does not directly supply the comparison used here.

\subsection{Orlicz balls, negative association, and product representations}
Two strands of work help explain why generalized Orlicz balls are a useful
test case. Pilipczuk and Wojtaszczyk \cite{PW} proved that the absolute
coordinates of a vector uniformly distributed on a generalized Orlicz ball
are negatively associated. This says, roughly, that increasing functions
of disjoint coordinate groups tend to be negatively correlated. It
captures the competition created by a shared constraint, but is not the
property used in the proof below.

A different strand seeks representations by independent random variables.
Barthe, Gu\'edon, Mendelson and Naor \cite{BGMN} represented the uniform
measure on an $\ell_r$ ball using independent generalized Gaussian
coordinates and an auxiliary exponential variable. For general Orlicz
balls, Kabluchko and Prochno \cite{KP} used the maximum-entropy viewpoint
and associated Gibbs distributions to study volume asymptotics.

These constructions motivate the product density
\[
 \prod_i c_{i,\tau}
       e^{-\lambda|v_i|-\tau\phi_i(|v_i|)}.
\]
Here, however, the required comparison is finite-dimensional and uniform:
we do not appeal to asymptotic equivalence of ensembles. An exponential
slack variable yields an \emph{exact} disintegration over the energy
budget. That identity is proved directly in
Section~\ref{sec:reference}.

\subsection{The particular mechanism of this proof}
The argument developed in the supplied manuscript \cite{BMM} replaces a
fixed radial scaling by a sequential coupling of budgets. The crucial
feature is not merely that a larger budget gives larger coordinates. Each
coordinate's increase in energy is bounded by the extra budget available
at that stage. This prevents the sequential construction from exhausting
more than the available difference.

There are then two opposing requirements. The independent reference must
be small enough that it rarely needs correction, but not so small that it
loses the moment separation of the original vector. We meet the first
requirement by increasing the exponential-tilt parameter, and the second
by a convex-order comparison after a universal dilation.

\begin{keypoint}
The product model is used for its minoration theorem. The budget coupling
is used to compare that model with the constrained vector. The small
correction is measured in expected seminorms, exactly the quantities
needed to compare widths.
\end{keypoint}
This historical discussion locates the method among earlier results; it
is not a claim of priority for every consequence or special case.

\section{The setting and the two principal theorems}\label{sec:setting}

\subsection{The family of measures}
Fix $d\ge1$, an energy budget $E>0$, and a parameter $\lambda\ge0$.
For each $i=1,\ldots,d$, assume that
\begin{equation}\label{eq:assumptions}
 \begin{gathered}
 \phi_i:[0,\infty)\longrightarrow[0,\infty)
       \text{ is finite, convex, and strictly increasing},\\
 \phi_i(0)=0,\qquad \lim_{r\to\infty}\phi_i(r)=\infty.
 \end{gathered}
\end{equation}
We use the term Orlicz function in this sense; in particular, linear
functions are allowed. Define
\begin{equation}\label{eq:measure}
 \Phi(x)=\sum_{i=1}^d\phi_i(|x_i|),\qquad
 d\mu_s(x)=\frac{1}{Z_s}e^{-\lambda\sum_i|x_i|}
                    \ind_{\{\Phi(x)\le s\}}\,dx,
 \qquad s>0.
\end{equation}
Let $X_s$ have distribution $\mu_s$, set $X=X_E$, and use the convention
$X_0=0$. Absolute values of vectors are always taken coordinatewise.

The support in \eqref{eq:measure} is compact, since
$|x_i|\le\phi_i^{-1}(s)$ there. It has nonempty interior. The function
$\Phi$ is convex, so the density is log-concave. It is also invariant
under every coordinatewise change of signs, that is, \emph{unconditional}.
Thus we may represent
\begin{equation}\label{eq:signrepresentation}
 X_s=\eps\odot U_s,
\end{equation}
where $U_s$ has the law of $|X_s|$, the coordinates of $\eps$ are independent
symmetric signs, and $\eps$ is independent of $U_s$.
The symbol $\odot$ denotes coordinatewise multiplication.

\subsection{What Sudakov minoration means here}
For an integrable symmetric vector $Y$ and a nonempty finite
$T\subset\R^d$, write
\[
 W_Y(T)=\E\max_{t\in T}\ip{t}{Y}.
\]
The vector is centered because it is symmetric. Hence translating $T$
does not change $W_Y(T)$.

\begin{definition}\label{def:smp}
The vector $Y$ satisfies $\SMP(\kappa)$ if, for every $p\ge1$, every finite
$T\subset\R^d$ with $|T|\ge e^p$, and every $a>0$, the separation
condition
\begin{equation}\label{eq:separation}
 \norm{\ip{t-u}{Y}}_p\ge a
 \qquad\text{for all distinct }t,u\in T
\end{equation}
implies $W_Y(T)\ge\kappa a$.
\end{definition}
Some references use the expected supremum of all increments instead of
$W_Y(T)$. For symmetric vectors, the former is $2W_Y(T)$, so this changes
only the numerical constant; see \eqref{eq:widthseminorm} below.

\begin{theorem}[Sudakov minoration]\label{thm:main}
Under \eqref{eq:assumptions}, the vector $X=X_E$ satisfies
$\SMP(\kappa)$ for a universal $\kappa>0$. The constant is independent
of $d$, $E$, $\lambda$, and all the functions $\phi_i$.
\end{theorem}
Neither isotropy nor a common upper-growth bound such as
$\phi_i(2r)\le C\phi_i(r)$ is assumed.

\subsection{The stronger comparison behind the theorem}
For integrable random vectors, write $Y\cx Z$ if
$\E f(Y)\le\E f(Z)$ for every convex function for which the comparison
is defined. This is \emph{convex order}. In the applications below the
vectors are centered; there is therefore no conflict with the fact that
convex order also compares affine functions in both directions.

\begin{theorem}[Comparison with one product distribution]\label{thm:comparison}
There is a universal $M\ge1$ and a parameter $\tau>0$ such that the vector
$V$ with independent coordinates and marginal densities
\begin{equation}\label{eq:product}
 f_{i,\tau}(v)=c_{i,\tau}
              \exp\{-\lambda|v|-\tau\phi_i(|v|)\},
 \qquad v\in\R,
\end{equation}
satisfies
\begin{equation}\label{eq:comparison}
 X\cx 2MV,
 \qquad
 \E q(V)\le2\E q(X)
 \quad\text{for every seminorm }q\text{ on }\R^d.
\end{equation}
The parameter $\tau$ depends on the distribution of $X$, but neither
$\tau$ nor $V$ depends on $T$, $p$, or $q$.
\end{theorem}
The two conclusions point in different directions. Convex order says that
linear moments of $X$ cannot exceed those of a universal dilation of $V$.
The seminorm comparison says that $V$ is not larger than $X$ at the level
of expected suprema. Both are needed.

\section{Why these comparisons transfer minoration}\label{sec:transfer}
We first isolate an elementary principle. Doing so makes the purpose of
each later estimate transparent.

\begin{lemma}[Transfer principle]\label{lem:transfer}
Let $X$ and $V$ be integrable symmetric vectors, and suppose $V$ satisfies
$\SMP(\kappa_0)$. Assume that, for constants $A,B>0$,
\begin{align}
 \norm{\ip{h}{X}}_p&\le A\norm{\ip{h}{V}}_p
 &&(h\in\R^d,\ p\ge1),\label{eq:transfermoment}\\
 \E q(V)&\le B\E q(X)
 &&(q\text{ a seminorm}).\label{eq:transfernorm}
\end{align}
Then $X$ satisfies $\SMP(\kappa_0/(AB))$.
\end{lemma}
\begin{proof}
Suppose $T,p,a$ satisfy Definition~\ref{def:smp} for $X$.
By \eqref{eq:transfermoment}, distinct points of $T$ are separated for
$V$ at level $a/A$. Thus
\[
 W_V(T)\ge\frac{\kappa_0a}{A}.
\]
To use \eqref{eq:transfernorm}, introduce the range seminorm
\begin{equation}\label{eq:qT}
 q_T(x)=\max_{t,u\in T}\ip{t-u}{x}
       =\max_{t\in T}\ip{t}{x}-\min_{u\in T}\ip{u}{x}.
\end{equation}
The set $T-T$ is symmetric and contains zero, so $q_T$ is nonnegative,
subadditive, and absolutely homogeneous. It is therefore a seminorm,
even when $T-T$ does not span the entire space.

For any integrable symmetric $Y$,
\begin{equation}\label{eq:widthseminorm}
 \E q_T(Y)
 =\E\max_{t\in T}\ip{t}{Y}
  +\E\max_{u\in T}\ip{u}{-Y}
 =2W_Y(T).
\end{equation}
Using \eqref{eq:transfernorm} with $q_T$ yields
$W_V(T)\le B W_X(T)$, and the result follows.
\end{proof}

Convex order $X\cx AV$ implies \eqref{eq:transfermoment}, because
$x\mapsto|\ip{h}{x}|^p$ is convex for $p\ge1$.
Thus Theorem~\ref{thm:comparison} gives Theorem~\ref{thm:main} with
$A=2M$ and $B=2$, once independent-coordinate minoration is available.

\subsection*{A roadmap for constructing the comparison}
We first construct a coupling that orders the magnitudes of $X_s$ as the
budget $s$ increases. Next, we express a product vector $V_\tau$ as a
mixture of these $X_s$ by adding exponential slack to its energy.
We choose a parameter for which $E$ is a median of the random budget,
and then increase that parameter by a universal factor $M$.

At this increased parameter, a budget above $E$ is a rare event.
The monotone coupling corrects $V_\tau$ to a vector $Q$ whose magnitude
is dominated by that of $X$. The residual $D=V_\tau-Q$ vanishes outside
the rare event. Sign contraction and a second-moment estimate give
\[
 \E q(D)\le\theta\,\E q(V_\tau),\qquad \theta<1.
\]
Since $V_\tau=Q+D$, the residual can be absorbed into the left-hand side.
This proves the expected-seminorm comparison. Finally, the median
calibration and a one-coordinate density comparison give the convex order
in the other direction.

\section{Three external inputs and an elementary sign lemma}\label{sec:tools}

\subsection{Log-concavity survives integration}
Pr\'ekopa's theorem \cite{Prekopa} states that if $F(x,y)$ is a nonnegative
log-concave function, then
\[
 x\longmapsto\int F(x,y)\,dy
\]
is log-concave whenever the integral is finite. We will apply this to
weighted volumes of sections of an Orlicz ball. The extra variable will
be the available budget, not another coordinate of the random vector.

\subsection{A seminorm cannot concentrate its mean on a very rare event}
A consequence of Borell's theory of log-concave measures \cite{Borell} is
that there exists a universal $C_B\ge1$ such that, for a log-concave vector
$Y$ and every seminorm $q$,
\begin{equation}\label{eq:borell}
 \bigl(\E q(Y)^2\bigr)^{1/2}\le C_B\E q(Y).
\end{equation}
Consequently, for any event $B$, whether or not independent of $Y$,
\begin{equation}\label{eq:raremean}
 \E[\ind_Bq(Y)]
 \le \sqrt{\PP(B)}\,\bigl(\E q(Y)^2\bigr)^{1/2}
 \le C_B\sqrt{\PP(B)}\,\E q(Y).
\end{equation}
This is the exact reason that a rare correction will be inexpensive.
Only the product vector will need to be log-concave here; the corrected
vector and the residual need not be.

\subsection{Minoration for independent log-concave coordinates}
We use the following established result: an independent vector with
symmetric log-concave coordinate distributions satisfies
$\SMP(\kappa_{\rm ind})$, where $\kappa_{\rm ind}>0$ is universal.
The log-concave-tail version is due to Lata\l a \cite{Latala97}; it is
restated as Theorem~14 in the arXiv version of \cite{LT}.

It applies to \eqref{eq:product}. Indeed, the positive tail is, up to a
constant,
\[
 t\longmapsto\int_0^\infty
       e^{-\lambda(t+u)-\tau\phi_i(t+u)}\,du,
 \qquad t\ge0,
\]
which is log-concave by Pr\'ekopa's theorem.
The coordinates may have different variances: dividing each coordinate
by its standard deviation and multiplying the corresponding index
coordinate by that standard deviation leaves the canonical process
unchanged. All relevant variances are finite and positive.

\subsection{Why coordinatewise contraction works after averaging signs}
\begin{lemma}[Sign contraction]\label{lem:sign}
If $0\le a_i\le b_i$ for $i=1,\ldots,d$, then, for every convex
$f:\R^d\to\R$,
\begin{equation}\label{eq:sign}
 \E_\eps f(\eps\odot a)\le\E_\eps f(\eps\odot b).
\end{equation}
\end{lemma}
\begin{proof}
Define the sign average
\[
 \bar f(x)=2^{-d}\sum_{\eps\in\{-1,1\}^d}f(\eps\odot x).
\]
It is convex and unconditional. With all other coordinates fixed,
it is an even convex function of the remaining coordinate. An even
convex function is nondecreasing on $[0,\infty)$: for $0\le a\le b$
with $b>0$, express $a$ as a convex combination of $b$ and $-b$.
Applying this argument one coordinate at a time proves
$\bar f(a)\le\bar f(b)$.
\end{proof}
In particular, a coupling with $|Y|\le|Z|$ coordinatewise, together with
common independent signs, implies $Y\cx Z$.
The averaging is important: a general seminorm need not be monotone
under a coordinatewise decrease of magnitudes before signs are averaged.

\section{Coupling the vectors at different budgets}\label{sec:budget}
The main structural fact is that the positive-coordinate laws can be
coupled increasingly in the budget. We prove a stronger energy statement,
which explains why the coordinate-by-coordinate construction is possible.

\subsection{Weighted section volumes}
For $J\subset\{1,\ldots,d\}$ and $b\ge0$, set
\begin{equation}\label{eq:W}
 W_J(b)=\int_{[0,\infty)^J}e^{-\lambda\sum_{j\in J}x_j}
          \ind_{\{\sum_{j\in J}\phi_j(x_j)\le b\}}\,dx_J.
\end{equation}
For negative $b$, put $W_J(b)=0$; for the empty set, use
$W_\varnothing(b)=1$ when $b\ge0$.
If $J$ is nonempty, then $W_J$ is positive and finite on $(0,\infty)$.
It is nondecreasing simply because increasing the budget enlarges the
integration region.

More importantly, $W_J$ is log-concave. The set
\[
 \left\{(x_J,b):x_j\ge0,\ \sum_{j\in J}\phi_j(x_j)\le b\right\}
\]
is convex. Its indicator, multiplied by
$e^{-\lambda\sum_{j\in J}x_j}$, is log-concave in $(x_J,b)$.
Integration in $x_J$ and Pr\'ekopa's theorem prove the assertion.
Thus
\begin{equation}\label{eq:Wproperties}
 W_J\text{ is nondecreasing and log-concave on }(0,\infty).
\end{equation}
These are different properties, and both will be used.

\subsection{Using energy rather than position as the coordinate}
Let $\psi_i=\phi_i^{-1}$. The assumptions imply that $\psi_i$ is increasing
and concave, with $\psi_i(0)=0$. Under the change of variables
$y=\phi_i(x)$, the one-coordinate weighted measure becomes
\begin{equation}\label{eq:w}
 e^{-\lambda x}\,dx=w_i(y)\,dy,
 \qquad w_i(y)=e^{-\lambda\psi_i(y)}\psi_i'(y),\quad y>0.
\end{equation}
The two factors in $w_i$ are nonnegative and nonincreasing. Hence $w_i$
has a nonincreasing version.

No differentiability assumption on $\phi_i$ is being added.
The derivative of the concave inverse is understood almost everywhere.
For $0<a<b$, concavity gives
$\psi_i(b)-\psi_i(a)=\int_a^b\psi_i'(y)\,dy$.
Letting $a\downarrow0$ gives the same formula on $[0,b]$, since
$\psi_i$ is continuous at zero and its derivative is nonnegative.
This proves the local absolute continuity needed for the substitution.
The density $w_i$ may be unbounded near zero; this causes no problem.

\subsection{One-coordinate comparison: two inequalities, not one}
Fix $i\in J$. Under the positive-orthant law with coordinate set $J$ and
budget $b>0$, let $Y_b=\phi_i(x_i)$. Integrating the remaining coordinates
gives its density
\begin{equation}\label{eq:conditional}
 g_b(y)=\frac{w_i(y)W_{J\setminus\{i\}}(b-y)}{W_J(b)},
 \qquad 0<y<b.
\end{equation}
For real random variables, $A\st B$ means stochastic domination,
equivalently $\PP(A>y)\le\PP(B>y)$ for every $y$.

\begin{lemma}[An energy coordinate is monotone and 1-Lipschitz in budget]
\label{lem:one}
For $b>0$ and $\Delta\ge0$,
\begin{equation}\label{eq:stochsandwich}
 Y_b\st Y_{b+\Delta}\st Y_b+\Delta.
\end{equation}
Consequently, the common-quantile coupling satisfies
\begin{equation}\label{eq:quantile}
 0\le Y_{b+\Delta}-Y_b\le\Delta\qquad\text{almost surely}.
\end{equation}
\end{lemma}
\begin{proof}
Write $W_-=W_{J\setminus\{i\}}$.

\proofstep{First inequality: a larger budget increases the energy.}
On the common support $0<y<b$,
\begin{equation}\label{eq:likelihoodbudget}
 \frac{g_{b+\Delta}(y)}{g_b(y)}
 =\frac{W_J(b)}{W_J(b+\Delta)}
  \frac{W_-(b+\Delta-y)}{W_-(b-y)}.
\end{equation}
For a concave function $\ell$, the increment
$\ell(z+\Delta)-\ell(z)$ is nonincreasing in $z$.
Take $\ell=\log W_-$ and $z=b-y$. The last ratio in
\eqref{eq:likelihoodbudget} is therefore nondecreasing in $y$.
Moreover, the extra support of $g_{b+\Delta}$ lies to the right of $b$.
The difference of the two densities can thus change sign only from
negative to positive. Since both have integral one, their distribution
functions are ordered, proving $Y_b\st Y_{b+\Delta}$.
When $J=\{i\}$, the same argument works with $W_-=1$.

\proofstep{Second inequality: the energy cannot increase by more than the extra budget.}
For $0\le y\le b$, substitute $u=v+\Delta$ in the upper-tail integral:
\begin{align*}
 \PP(Y_{b+\Delta}>y+\Delta)
 &=\frac{\int_y^b w_i(v+\Delta)W_-(b-v)\,dv}{W_J(b+\Delta)}\\
 &\le\frac{\int_y^b w_i(v)W_-(b-v)\,dv}{W_J(b)}
 =\PP(Y_b>y).
\end{align*}
Here $w_i$ is nonincreasing and $W_J$ is nondecreasing.
Outside the displayed range, the desired tail inequality is immediate
from the supports. This proves $Y_{b+\Delta}\st Y_b+\Delta$.

\proofstep{Putting the two comparisons in one coupling.}
If $F_b^{-1}$ denotes the generalized inverse distribution function, then
\eqref{eq:stochsandwich} gives
\[
 F_b^{-1}(u)\le F_{b+\Delta}^{-1}(u)
            \le F_b^{-1}(u)+\Delta,\qquad 0<u<1.
\]
Apply all three functions to the same uniform random variable.
\end{proof}

\begin{keypoint}
The first stochastic inequality alone would not suffice. In the sequential
construction, we must also preserve the order of the \emph{remaining}
budgets. The upper bound by $\Delta$ is what guarantees that order.
\end{keypoint}

\subsection{The sequential coupling}
\begin{proposition}[Monotone budget coupling]\label{prop:budget}
For $0\le s\le t$, there is a coupling of
$U_s\sim\Law(|X_s|)$ and $U_t\sim\Law(|X_t|)$ such that
\begin{equation}\label{eq:budget}
 U_s\le U_t\quad\text{coordinatewise},\qquad
 0\le\Phi(U_t)-\Phi(U_s)\le t-s.
\end{equation}
The construction is measurable in $(s,t)$ and may be chosen to be the
identity when $s=t$.
\end{proposition}
\begin{proof}
Sample coordinates in the order $1,\ldots,d$, using the same independent
uniform variables for the two vectors. Initially the remaining budgets
are $s$ and $t$. Suppose at some stage they are $b$ and $b+\Delta$.
Conditioning on the previously chosen coordinates leaves exactly the
positive-orthant law for the remaining coordinate set and the remaining
budget, because both the weight and the constraint are additive.

Use Lemma~\ref{lem:one} to select the next two energies. Their difference
$\delta$ satisfies $0\le\delta\le\Delta$. After those energies are
subtracted, the remaining budgets are still ordered, and their difference
is $\Delta-\delta$. Induction proves coordinatewise order of every energy,
and hence of every magnitude.

Writing $\Delta_k$ for the difference of remaining budgets before
coordinate $k$, and $\delta_k$ for the energy difference selected at that
coordinate, we have
\[
 \Delta_{k+1}=\Delta_k-\delta_k,\qquad
 0\le\delta_k\le\Delta_k,\qquad \Delta_1=t-s.
\]
Therefore
\[
 \Phi(U_t)-\Phi(U_s)=\sum_{k=1}^d\delta_k
                   =t-s-\Delta_{d+1}\le t-s.
\]
For $s=0$, take $U_s=0$.

The conditional distribution functions in \eqref{eq:conditional} are
measurable in their budget and coordinate arguments. Their generalized
inverses are measurable as well: the event that an inverse is at most
$y$ is expressed through the corresponding distribution function at
$y$. Thus the finite recursion is measurable in the budgets and the
uniform variables. Equal budgets use the same quantiles and give the
identity coupling.
\end{proof}

Adding common independent signs and applying Lemma~\ref{lem:sign} yields
\begin{equation}\label{eq:budgetcx}
 X_s\cx X_t\qquad(0\le s\le t).
\end{equation}
The same construction provides a measurable downward correction from an
already sampled vector at budget $t$ to the law at budget $s$.
One way to see this is to recover the successive uniform variables by
applying the conditional distribution functions to the sampled energies,
and then run the smaller-budget recursion with those uniforms.
The conditional laws have densities, so this quantile recovery is valid
almost surely. Values on null boundary cases may be defined arbitrarily.

\section{An independent reference and an exact slack identity}\label{sec:reference}

\subsection{Softening the constraint}
For each $\tau>0$, let $V_\tau$ have the product density
\eqref{eq:product}. The parameter $\tau$ penalizes the total energy:
\[
 f_{V_\tau}(v)=\frac{1}{\mathcal Z_\tau}
                e^{-\lambda\sum_i|v_i|-\tau\Phi(v)}.
\]
This density is integrable. Indeed, a convex strictly increasing function
that vanishes at zero has an at-least-linear lower bound at infinity.
Each marginal therefore has exponential decay or faster, and all its
moments are finite.

The product law is easier to analyze than $\mu_E$, but its energy is not
fixed. To express that randomness in a useful way, independently sample
an exponential variable $H_\tau$ of rate $\tau$, meaning that
$\PP(H_\tau>h)=e^{-\tau h}$ for $h\ge0$, and set
\begin{equation}\label{eq:slack}
 S_\tau=\Phi(V_\tau)+H_\tau.
\end{equation}
We call $H_\tau$ the slack and $S_\tau$ the random budget.

\begin{lemma}[Exact disintegration]\label{lem:slack}
Let $W=W_{\{1,\ldots,d\}}$. Then a version of the conditional law is
\begin{equation}\label{eq:disintegration}
 \Law(V_\tau\mid S_\tau=s)=\mu_s,
 \qquad
 f_\tau(s)=\frac{W(s)e^{-\tau s}}{L(\tau)},\quad s>0,
\end{equation}
where $L(\tau)=\int_0^\infty W(u)e^{-\tau u}\,du$.
\end{lemma}
\begin{proof}
In the joint density of $(V_\tau,H_\tau)$, make the substitution
$s=\Phi(v)+h$. The Jacobian in the slack variable is one. The result is
\begin{align*}
 f_{V_\tau,S_\tau}(v,s)
 &=\frac{\tau}{\mathcal Z_\tau}
   e^{-\lambda\sum_i|v_i|-\tau\Phi(v)}
   e^{-\tau(s-\Phi(v))}\ind_{\{\Phi(v)\le s\}}\\
 &=\frac{\tau}{\mathcal Z_\tau}
   e^{-\tau s}e^{-\lambda\sum_i|v_i|}
   \ind_{\{\Phi(v)\le s\}}.
\end{align*}
For fixed $s$, the dependence on $v$ is exactly the unnormalized density
of $\mu_s$. Integrating in $v$ gives $Z_s=2^dW(s)$ and hence a density
of $S_\tau$ proportional to $W(s)e^{-\tau s}$.
Normalization gives \eqref{eq:disintegration}.
\end{proof}

All the integrals in this identity are finite. For example,
\[
 W(s)\le\prod_{i=1}^d\psi_i(s)\le C(1+s)^d,
\]
where $C$ may depend on the functions $\phi_i$.
The second bound follows from concavity of each $\psi_i$ and its value
at one. The exponential factor makes $L(\tau)$ finite for every $\tau>0$.

\subsection{Why the slack matters}
Conditioning on $\Phi(V_\tau)=s$ would give a law on an energy surface,
not the desired law throughout the sublevel set. In contrast, conditioning
on $\Phi(V_\tau)+H_\tau=s$ leaves all points with $\Phi(v)\le s$
available. The two exponential energy terms cancel exactly.

Thus the relation between the product law and the constrained laws is
not an approximation:
\[
 \Law(V_\tau)=\int_0^\infty\mu_s\,f_\tau(s)\,ds.
\]
Moreover, increasing $\tau$ simply tilts the one-dimensional budget
density towards smaller values. This is the source of a calibration
that works uniformly in the dimension and in the Orlicz functions.

\section{Calibration: a median and a rare correction event}\label{sec:calibration}

\subsection{Choosing the initial parameter}
Choose $\tau_0>0$ so that
\begin{equation}\label{eq:median}
 \PP(S_{\tau_0}\le E)=\frac12.
\end{equation}
We justify existence rather than treating it as an implicit regularity
assumption. From \eqref{eq:disintegration}, the probability is continuous
in $\tau>0$, by dominated convergence on compact parameter intervals.
Increasing $\tau$ multiplies the density by a decreasing exponential
function of $s$, so this probability is nondecreasing.

As $\tau\downarrow0$, the pathwise inequality $S_\tau\ge H_\tau$ gives
\[
 \PP(S_\tau\le E)\le1-e^{-\tau E}\longrightarrow0.
\]
As $\tau\to\infty$, the probability tends to one. To verify the latter
claim directly, let
$A=\int_{E/4}^{E/2}W(s)\,ds>0$. For $\tau\ge1$,
\begin{align*}
 L(\tau)&\ge e^{-\tau E/2}A,\\
 \int_E^\infty W(s)e^{-\tau s}\,ds
 &\le e^{-(\tau-1)E}\int_E^\infty W(s)e^{-s}\,ds.
\end{align*}
The ratio tends to zero. The intermediate value theorem now proves
\eqref{eq:median}. Since $S_{\tau_0}$ has a density, its upper tail at
$E$ also has probability $1/2$.

\subsection{Increasing the parameter makes correction rare}
\begin{lemma}[Rare-budget estimate]\label{lem:rare}
For every $M\ge1$,
\begin{equation}\label{eq:rare}
 \PP(S_{M\tau_0}>E)\le\eta_M,
 \qquad \eta_M=2(2/3)^{M-1}.
\end{equation}
\end{lemma}
\begin{proof}
Write $f_0=f_{\tau_0}$ and
$\overline F_0(s)=\PP(S_{\tau_0}>s)$.
We first show that a fixed amount of probability lies a definite distance
below the median, where the distance is measured in units of $1/\tau_0$.

\proofstep{A bound on the hazard.}
Because $W$ is nondecreasing, for $u\ge s>0$,
\[
 f_0(u)\ge f_0(s)e^{-\tau_0(u-s)}.
\]
Integrating gives
\[
 \overline F_0(s)\ge\frac{f_0(s)}{\tau_0}.
\]
The hazard, namely the density divided by the remaining upper-tail
probability, therefore satisfies
\begin{equation}\label{eq:hazard}
 \frac{f_0(s)}{\overline F_0(s)}\le\tau_0.
\end{equation}

\proofstep{A quarter of the mass is separated from the median.}
Let $h=\log(3/2)/\tau_0$.
The inequalities $S_{\tau_0}\ge H_{\tau_0}$ and
$\PP(S_{\tau_0}>E)=1/2$ imply $\tau_0E\ge\log2$, so $h<E$.
Since the derivative of $-\log\overline F_0$ is the hazard, integrating
\eqref{eq:hazard} from $E-h$ to $E$ yields
\[
 \overline F_0(E-h)
 \le e^{\tau_0h}\overline F_0(E)
 =\frac34.
\]
Consequently,
\begin{equation}\label{eq:quarter}
 \PP(S_{\tau_0}\le E-h)\ge\frac14.
\end{equation}

\proofstep{Exponential tilting separates the two regions.}
Put $a=(M-1)\tau_0$. By \eqref{eq:disintegration}, the law at
$M\tau_0$ is the exponential tilt of the law at $\tau_0$:
\begin{align*}
 \PP(S_{M\tau_0}>E)
 &=\frac{\E[e^{-aS_{\tau_0}}\ind_{\{S_{\tau_0}>E\}}]}
         {\E e^{-aS_{\tau_0}}}\\
 &\le\frac{\frac12 e^{-aE}}
              {\frac14 e^{-a(E-h)}}
 =2e^{-ah}
 =2(2/3)^{M-1}.
\end{align*}
The numerator uses the median identity, and the denominator uses
\eqref{eq:quarter}.
\end{proof}

The estimate does not involve a concentration parameter for the total
energy. It needs only the monotonicity of $W$ and the exact exponential
family of budget densities. In particular, the universal exponential
decay in $M$ does not deteriorate when $d$ grows.

\section{Local correction and absorption of the residual}\label{sec:correction}

\subsection{Constructing a vector that respects the constraint}
Fix $M\ge1$, and abbreviate
\[
 \tau=M\tau_0,\qquad V=V_\tau,\qquad S=S_\tau.
\]
Conditionally on $S=s$, the magnitude of $V$ has the law of $|X_s|$.
Apply Proposition~\ref{prop:budget} to correct this magnitude down to
budget $\min(s,E)$. Preserve the original independent signs, and call
the resulting vector $Q$. Define
\begin{equation}\label{eq:D}
 D=V-Q.
\end{equation}
The correction uses only magnitudes and budgets, never the signs.

\begin{proposition}[The corrected vector]\label{prop:correction}
The construction satisfies
\begin{equation}\label{eq:local}
 \begin{gathered}
 |Q_i|\le|V_i|,\qquad \Phi(Q)\le E,
 \qquad Q=V\ \text{on }\{S\le E\},\\
 0\le\Phi(V)-\Phi(Q)\le(S-E)_+.
 \end{gathered}
\end{equation}
Moreover,
\begin{equation}\label{eq:Qlaw}
 \Law(Q\mid S=s)=\mu_{\min(s,E)},\qquad Q\cx X.
\end{equation}
\end{proposition}
\begin{proof}
Every pathwise assertion follows from the ordered budget coupling,
including its identity property when the two budgets coincide.
The conditional marginal of the corrected vector is exactly
$\mu_{\min(s,E)}$. This law is dominated by $\mu_E$ in the magnitude
coupling of Proposition~\ref{prop:budget}. Adding signs and using
Lemma~\ref{lem:sign}, followed by integration in $s$, gives $Q\cx X$.
\end{proof}

Two distinctions are worth keeping in mind. First, $Q$ does not generally
have the same law as $X$; it is a mixture of laws with budgets at most
$E$. Convex-order domination is enough. Second, the event we control is
$\{S>E\}$, not just $\{\Phi(V)>E\}$. The slack may make $S>E$ even when
$V$ already lies inside the constraint. The construction is designed to
control the law of $Q$, not to be a nearest-point projection.

\subsection{The same coupling controls all coordinate caps}
For $\beta\in(0,\infty)^d$, write
\[
 (\clip_\beta x)_i=\operatorname{sgn}(x_i)\min(|x_i|,\beta_i).
\]
\begin{corollary}[All-cap convex-order comparisons]\label{cor:caps}
For every $\beta\in(0,\infty)^d$ and every $L\ge1$,
\begin{equation}\label{eq:caps}
 \clip_\beta Q\cx\clip_\beta X,
 \qquad
 \clip_{L\beta}Q\cx L\clip_\beta X.
\end{equation}
\end{corollary}
\begin{proof}
Use the magnitude coupling $|Q|\le|X|$ available in the proof of
Proposition~\ref{prop:correction}, with common independent signs.
Coordinatewise,
\[
 \min(|Q_i|,\beta_i)\le\min(|X_i|,\beta_i),
\]
and
\[
 \min(|Q_i|,L\beta_i)\le L\min(|X_i|,\beta_i).
\]
Apply sign contraction to each comparison.
\end{proof}
Here $L$ is a cap-dilation parameter; it may be taken equal to the
calibration factor $M$. These statements come from the underlying ordered
coupling, not from a general rule that nonlinear clipping preserves
convex order.

\subsection{The residual estimate}
\begin{proposition}[Small residual in every expected seminorm]\label{prop:residual}
For every seminorm $q$ on $\R^d$,
\begin{equation}\label{eq:residual}
 \E q(D)\le\theta_M\E q(V),
 \qquad
 \theta_M=C_B\sqrt2\,(2/3)^{(M-1)/2}.
\end{equation}
\end{proposition}
\begin{proof}
Let $B=\{S>E\}$. Conditional on all magnitudes, the slack, and any
randomness used in the correction, write
\[
 V=\eps\odot v,\qquad Q=\eps\odot u,\qquad
 D=\eps\odot(v-u),
\]
where $0\le u\le v$ coordinatewise. The signs are still independent and
symmetric. Moreover, $v-u=0$ off $B$. Thus Lemma~\ref{lem:sign} gives,
conditionally,
\[
 \E_\eps q\bigl(\eps\odot(v-u)\bigr)
 \le\ind_B\,\E_\eps q(\eps\odot v).
\]
After averaging the remaining variables, this becomes
\begin{equation}\label{eq:signresidual}
 \E q(D)\le\E[\ind_Bq(V)].
\end{equation}
Now use Cauchy--Schwarz, the log-concavity of $V$, and
Lemma~\ref{lem:rare}:
\begin{align*}
 \E q(D)
 &\le\sqrt{\PP(B)}\,\bigl(\E q(V)^2\bigr)^{1/2}\\
 &\le C_B\sqrt{\eta_M}\,\E q(V)
 =C_B\sqrt2\,(2/3)^{(M-1)/2}\E q(V).
\end{align*}
\end{proof}

\begin{remark}[Why the conditional sign average cannot be omitted]
The pointwise implication $|r_i|\le|v_i|\Rightarrow q(r)\le q(v)$ is
false for an arbitrary seminorm. For example, take
$q(x)=|x_1-x_2|$, $v=(1,1)$, and $r=(1,0)$. Then $q(v)=0$ but
$q(r)=1$. The valid statement used above is the contraction after
averaging independent signs. Also, no independence between $B$ and $V$
is assumed in \eqref{eq:signresidual} or in Cauchy--Schwarz.
\end{remark}

\subsection{Absorbing the error}
Choose a universal $M$ large enough that $\theta_M\le1/2$.
For definiteness, the choice
\begin{equation}\label{eq:M}
 M=1+\left\lceil\frac{\log(8C_B^2)}{\log(3/2)}\right\rceil
\end{equation}
suffices. We do not need an optimized value.
The triangle inequality, $Q\cx X$, and \eqref{eq:residual} imply
\begin{align*}
 \E q(V)
 &\le\E q(Q)+\E q(D)\\
 &\le\E q(X)+\theta_M\E q(V).
\end{align*}
Moving the last term to the left gives
\begin{equation}\label{eq:absorb}
 \E q(V)\le\frac{1}{1-\theta_M}\E q(X)\le2\E q(X).
\end{equation}
This is the second half of Theorem~\ref{thm:comparison}.
The choice of $M$ is the same for every seminorm, every dimension, and
every set of Orlicz functions.

\section{Convex order in the opposite direction}\label{sec:reverse}
The reference at parameter $M\tau_0$ now has a small expected correction.
We must still verify that it retains the moments required for separation.
This is done in two stages: use the median at $\tau_0$, then compare the
two product distributions.

\begin{lemma}[Reverse comparison]\label{lem:reverse}
For every $M\ge1$,
\begin{equation}\label{eq:reverse}
 X\cx2V_{\tau_0}\cx2MV_{M\tau_0}.
\end{equation}
\end{lemma}
\begin{proof}
\proofstep{The median mixture dominates half of the desired convex expectation.}
Let $f:\R^d\to\R$ be convex, and let $\bar f$ be its sign average.
All vectors in the assertion are unconditional, so their expectations
are unchanged by replacing $f$ with $\bar f$. This function is convex,
unconditional, and minimized at zero.

When $s\ge E$, \eqref{eq:budgetcx} gives
$\E\bar f(X_s)\ge\E\bar f(X)$. For $s<E$, the lower bound
$\bar f(0)$ is sufficient. Disintegration and the median identity therefore
yield
\begin{equation}\label{eq:half}
 \E\bar f(V_{\tau_0})
 \ge\frac12\E\bar f(X)+\frac12\bar f(0).
\end{equation}
Convexity gives
$\bar f(x)\le\tfrac12\bar f(2x)+\tfrac12\bar f(0)$.
Combining this with \eqref{eq:half}, we obtain
\[
 \E\bar f(2V_{\tau_0})\ge\E\bar f(X),
\]
which proves $X\cx2V_{\tau_0}$.

\proofstep{A universal dilation compensates for increasing the parameter.}
Put $\tau=M\tau_0$ and consider one positive coordinate. The density of
$M|V_{\tau,i}|$, divided by the density of $|V_{\tau_0,i}|$, is a positive
constant times
\begin{equation}\label{eq:scaledratio}
 \exp\left\{\lambda(1-1/M)x
       +\tau_0\phi_i(x)-M\tau_0\phi_i(x/M)\right\},
 \qquad x>0.
\end{equation}
Its logarithm is locally absolutely continuous. Its almost-everywhere
derivative equals
\[
 \lambda(1-1/M)
   +\tau_0\bigl(\phi_i'(x)-\phi_i'(x/M)\bigr)\ge0,
\]
because a convex function has a nondecreasing derivative.
The ratio is therefore nondecreasing. The one-crossing argument used in
Lemma~\ref{lem:one} implies
\begin{equation}\label{eq:coordinatecooling}
 |V_{\tau_0,i}|\st M|V_{\tau,i}|.
\end{equation}
Use independent common-quantile couplings in the different coordinates;
their product preserves the independent-coordinate marginal laws.
After adding common independent signs, Lemma~\ref{lem:sign} gives
$V_{\tau_0}\cx MV_\tau$. Dilation by two completes
\eqref{eq:reverse}.
\end{proof}

For convex functions with possibly infinite expectations, subtract an
affine minorant first. Alternatively, after sign averaging use the
nonnegative function $\bar f-\bar f(0)$. The preceding inequalities are
then meaningful in the extended sense. The moment applications below
are all finite.

The reverse comparison together with \eqref{eq:absorb} proves
Theorem~\ref{thm:comparison}. In particular, for $V=V_{M\tau_0}$,
\begin{equation}\label{eq:moments}
 \norm{\ip{h}{X}}_p\le2M\norm{\ip{h}{V}}_p,
 \qquad h\in\R^d,\quad p\ge1.
\end{equation}
Notice what was \emph{not} needed: no estimate of the form
$\phi_i(Mx)\le C(M)\phi_i(x)$ appears. Only monotonicity of the derivative
of a convex function is used in \eqref{eq:scaledratio}.

\section{Completion of the Sudakov minoration proof}\label{sec:completion}
\begin{proof}[Proof of Theorem~\ref{thm:main}]
Let $T,p,a$ satisfy Definition~\ref{def:smp} for $X$.
The comparison \eqref{eq:moments} yields
\[
 \norm{\ip{t-u}{V}}_p\ge\frac{a}{2M}
 \qquad(t,u\in T,\ t\ne u).
\]
The coordinates of $V$ are independent, symmetric, and log-concave.
Their minoration theorem gives
\[
 W_V(T)\ge\frac{\kappa_{\rm ind}}{2M}\,a.
\]
Apply \eqref{eq:absorb} to the range seminorm $q_T$ from
\eqref{eq:qT}. By \eqref{eq:widthseminorm},
\[
 2W_V(T)=\E q_T(V)\le2\E q_T(X)=4W_X(T).
\]
Therefore
\begin{equation}\label{eq:final}
 \boxed{\qquad
 W_X(T)\ge\frac12W_V(T)
          \ge\frac{\kappa_{\rm ind}}{4M}\,a.
 \qquad}
\end{equation}
The $M$ in \eqref{eq:M} is universal, which proves the theorem with
$\kappa=\kappa_{\rm ind}/(4M)$.
\end{proof}

\subsection*{What each part of the argument contributed}
The budget coupling provided the distributional comparison
$Q\cx X$ while respecting the constraint. The slack identity selected
the appropriate constrained law at every random budget. Median
calibration had two uses: it supplied the convex-order lower comparison
at $\tau_0$, and it made the correction event exponentially rare after
increasing the parameter. Finally, sign contraction and Borell's bound
converted rarity into a small expected seminorm.

Thus the proof does not try to dominate every residual linear moment or
every selected copy separately. It obtains exactly the two comparisons
needed in Lemma~\ref{lem:transfer}, with constants uniform in the moment
order and in the index set.

\section{Examples and concrete interpretations}\label{sec:examples}

\subsection{Uniform generalized Orlicz balls}
When $\lambda=0$, the measure is uniform on
\[
 K_E=\left\{x\in\R^d:\sum_{i=1}^d\phi_i(|x_i|)\le E\right\}.
\]
Theorem~\ref{thm:main} applies without requiring the functions to coincide.
Taking $\phi_i(u)=u^r$ for a fixed $r\ge1$ gives an ordinary
$\ell_r$ ball. These ordinary-ball cases are already contained in
\cite[Corollary~5.3]{Latala14}. The purpose of the general construction is
to avoid relying on a single radial homogeneity.

\subsection{Mixed powers: a fully explicit budget distribution}
Take
\begin{equation}\label{eq:mixedpower}
 \phi_i(u)=c_i u^{r_i},\qquad c_i>0,\quad r_i\ge1,
 \qquad \lambda=0.
\end{equation}
The exponents may differ from coordinate to coordinate and need not have
a common upper bound. Put
\[
 \alpha_i=\frac1{r_i},\qquad \alpha=\sum_{i=1}^d\alpha_i.
\]
The inverse and energy density factor are
\[
 \psi_i(y)=c_i^{-1/r_i}y^{1/r_i},\qquad
 w_i(y)=\frac{c_i^{-1/r_i}}{r_i}y^{\alpha_i-1}.
\]
Changing variables $y_i=s z_i$ in the positive-orthant volume shows
\begin{equation}\label{eq:powerW}
 W(s)=C s^\alpha
\end{equation}
for a constant $C>0$ depending on the coefficients and exponents.
Hence \eqref{eq:disintegration} becomes
\begin{equation}\label{eq:gamma}
 f_\tau(s)=\frac{\tau^{\alpha+1}}{\Gamma(\alpha+1)}
                 s^\alpha e^{-\tau s},\qquad s>0.
\end{equation}
Thus $S_\tau$ has a gamma distribution of shape $\alpha+1$ and rate
$\tau$. The same fact is visible directly: the independent energies
$c_i|V_{\tau,i}|^{r_i}$ have gamma distributions of shapes $\alpha_i$
and common rate $\tau$, while the slack contributes an additional
shape parameter one.

If $m_{\alpha+1}$ is the median of the gamma distribution with shape
$\alpha+1$ and rate one, then the calibration is simply
\[
 \tau_0=\frac{m_{\alpha+1}}{E}.
\]
For example, the two-dimensional region
$|x_1|+|x_2|^2\le E$ has $\alpha=3/2$, so the random budget has gamma
shape $5/2$.

There is even a direct budget coupling in this particular example class.
If $U_1$ is uniform on the positive part of $K_1$, set
$(U_s)_i=s^{1/r_i}(U_1)_i$. The constant Jacobian of this diagonal scaling
shows that $U_s$ has the correct law. Moreover,
\[
 \Phi(U_t)-\Phi(U_s)=(t-s)\Phi(U_1)\le t-s.
\]
For arbitrary Orlicz functions such a coordinatewise power scaling is
not available; the quantile construction replaces it.

\subsection{Quadratic--linear energies and rapidly growing energies}
Another admissible family is
\[
 \Phi(x)=\sum_{i=1}^d\bigl(a_i|x_i|+b_i x_i^2\bigr),
 \qquad a_i,b_i\ge0,\quad a_i+b_i>0.
\]
Different coordinates may be mostly linear, mostly quadratic, or a
mixture of the two. The reference density is explicit up to its
one-coordinate normalizing constants:
\[
 f_{i,\tau}(v)\propto
       e^{-(\lambda+\tau a_i)|v|-\tau b_i v^2}.
\]
Both the uniform measure and its exponential weight
$e^{-\lambda\sum_i|x_i|}$ are covered.

The functions may also grow much faster than a fixed power. For example,
\[
 \phi_i(u)=e^{u/c_i}-1,\qquad c_i>0,
\]
satisfy the assumptions, although the ratio
$\phi_i(2u)/\phi_i(u)$ is unbounded as $u\to\infty$.
This illustrates why avoiding an upper-growth hypothesis is substantive,
not just a notational simplification.

\subsection{Linear changes of coordinates and isotropic position}
No normalization of the covariance matrix was used. More generally,
if $A$ is invertible, then
\[
 \ip{t-u}{AX}=\ip{A^{\mathsf T}(t-u)}{X},\qquad
 W_{AX}(T)=W_X(A^{\mathsf T}T).
\]
The transformed index set has the same cardinality. Therefore SMP is
preserved, with the same constant, by invertible linear transformations.
In particular, one may subsequently place the vector in isotropic
position. This is a consequence of the theorem, not an extra assumption
needed to run its proof.

\section{The exact scope of the conclusion}\label{sec:scope}

\subsection{Why negative association is not the hypothesis being used}
For uniform generalized Orlicz balls, negative association of the absolute
coordinates is known \cite{PW}. The present proof uses more specific
information: the complete family of sublevel laws, an ordered coupling
between its budgets, and an exact product/slack representation.
It does not deduce these ingredients from negative association alone.
Thus the argument gives no automatic extension to an arbitrary
negatively associated unconditional log-concave vector.

The shared constraint is also genuinely scalar. With several overlapping
constraints, the remaining feasible region need not be described by one
number, and the recursion
$\Delta_{k+1}=\Delta_k-\delta_k$ has no immediate analogue.
This proof does not supply the corresponding multiconstraint coupling.

\subsection{What the residual estimate controls}
The residual in \eqref{eq:residual} is exactly $D=V-Q$, before any extra
clipping. The estimate is an unconditional expectation inequality for
every deterministic seminorm. It does not assert a small residual under
an arbitrary change of probability measure, under a Gibbs selector, or
for an empirical law of conditional residual means.

Nor does it imply the same small-error bound for
$V-\clip_{M\beta}Q$ at arbitrary caps. A one-dimensional calculation shows
why such a claim would be too strong without additional restrictions:
\[
 \E|V-\clip_{M\beta}Q|\ge\E|V|-M\beta.
\]
As $\beta\downarrow0$, the right-hand side tends to $\E|V|>0$.
Thus no upper bound by $\theta\E|V|$ with fixed $\theta<1$ can hold
uniformly for all positive caps. This does not conflict with
\eqref{eq:caps}, which compares the \emph{clipped sources}, not the full
residual after clipping.

\subsection{Assumptions that have deliberately been retained}
Finite, strictly increasing Orlicz functions ensure that the inverse
energy coordinates are well defined and that the conditional laws used
in the quantile recursion have densities. Functions with flat pieces or
infinite-valued hard barriers are not included in the theorem as stated.
The argument above does not rely on a limiting extension to those cases.

\begin{keypoint}
The proved conclusion is the universal comparison
\eqref{eq:comparison}, and therefore SMP, for the specific densities
\eqref{eq:measure}. The product vector is chosen once for the distribution,
and the residual is small in every expected seminorm. No stronger
selector-dependent estimate is needed to complete this proof.
\end{keypoint}

\section*{Notation at a glance}
\addcontentsline{toc}{section}{Notation at a glance}
\renewcommand{\arraystretch}{1.3}
\begin{center}
\begin{tabular}{@{}p{0.20\textwidth}p{0.73\textwidth}@{}}
\toprule
Symbol & Meaning\\
\midrule
$\Phi(x)$ & Total energy $\sum_i\phi_i(|x_i|)$.\\
$\mu_s$, $X_s$ & Constrained law at budget $s$ and a vector with that law.\\
$X$ & The vector of interest, $X_E$.\\
$W_Y(T)$ & Expected maximum $\E\max_{t\in T}\ip{t}{Y}$.\\
$W_J(b)$ & Positive-orthant weighted volume with coordinates $J$ and budget $b$.\\
$V_\tau$ & Independent reference with densities \eqref{eq:product}.\\
$H_\tau$, $S_\tau$ & Exponential slack and the random budget $\Phi(V_\tau)+H_\tau$.\\
$\tau_0$, $M$ & Median-calibrated parameter and its universal multiplier.\\
$Q$, $D$ & Corrected vector and residual $D=V-Q$.\\
$\cx$, $\st$ & Convex order and one-dimensional stochastic order.\\
$C_B$, $\kappa_{\rm ind}$ & Universal constants in Borell's bound and independent SMP.\\
\bottomrule
\end{tabular}
\end{center}


\clearpage
\begin{thebibliography}{99}
\addcontentsline{toc}{section}{References}
\small

\bibitem{BGMN}
F. Barthe, O. Gu\'edon, S. Mendelson and A. Naor,
\emph{A probabilistic approach to the geometry of the $\ell_p^n$-ball},
Annals of Probability \textbf{33} (2005), no.~2, 480--513.
\href{https://doi.org/10.1214/009117904000000874}{doi:10.1214/009117904000000874}.

\bibitem{Bednorz14}
W. Bednorz, \emph{Some remarks on the Sudakov minoration},
arXiv preprint, 2014, version~2.
\href{https://arxiv.org/abs/1404.6045v2}{arXiv:1404.6045v2}.

\bibitem{BL}
W. Bednorz and R. Lata\l a,
\emph{On the boundedness of Bernoulli processes},
Annals of Mathematics \textbf{180} (2014), no.~3, 1167--1203.
\href{https://doi.org/10.4007/annals.2014.180.3.8}{doi:10.4007/annals.2014.180.3.8}.

\bibitem{BMM}
W. Bednorz, R. Martynek and R. Meller,
\emph{Sudakov minoration for weighted Orlicz balls},
2026. Supplied manuscript \texttt{Weighted\_Orlicz\_Sudakov\_Proof-2.tex};
see also \href{https://arxiv.org/abs/2609.19908}{arXiv:2609.19908}.
The present text is an expanded expository version of that argument.

\bibitem{Borell}
C. Borell, \emph{Convex measures on locally convex spaces},
Arkiv f\"or Matematik \textbf{12} (1974), 239--252.
\href{https://doi.org/10.1007/BF02384761}{doi:10.1007/BF02384761}.

\bibitem{KP}
Z. Kabluchko and J. Prochno,
\emph{The maximum entropy principle and volumetric properties of Orlicz balls},
arXiv version, 2020.
\href{https://arxiv.org/abs/2007.05247v2}{arXiv:2007.05247v2}.

\bibitem{Latala97}
R. Lata\l a, \emph{Sudakov minoration principle and supremum of some processes},
Geometric and Functional Analysis \textbf{7} (1997), 936--953.
\href{https://doi.org/10.1007/s000390050031}{doi:10.1007/s000390050031}.

\bibitem{Latala14}
R. Lata\l a, \emph{Sudakov-type minoration for log-concave vectors},
Studia Mathematica \textbf{223} (2014), 251--274.
\href{https://arxiv.org/abs/1311.6428v2}{arXiv:1311.6428v2}.

\bibitem{LT}
R. Lata\l a and T. Tkocz,
\emph{A note on suprema of canonical processes based on random variables
with regular moments}, Electronic Journal of Probability \textbf{20}
(2015), no.~36, 1--17.
\href{https://doi.org/10.1214/EJP.v20-3625}{doi:10.1214/EJP.v20-3625}.
For the theorem numbering used here, see
\href{https://arxiv.org/abs/1406.6584v1}{arXiv:1406.6584v1}.

\bibitem{PW}
M. Pilipczuk and J. O. Wojtaszczyk,
\emph{The negative association property for the absolute values of random
variables equidistributed on a generalized Orlicz ball},
Positivity \textbf{12} (2008), 421--474.
\href{https://doi.org/10.1007/s11117-007-2116-4}{doi:10.1007/s11117-007-2116-4}.
\href{https://arxiv.org/abs/0803.0434}{arXiv:0803.0434}.

\bibitem{Prekopa}
A. Pr\'ekopa, \emph{On logarithmic concave measures and functions},
Acta Scientiarum Mathematicarum (Szeged) \textbf{34} (1973), 335--343.
\href{https://rutcor.rutgers.edu/~prekopa/SCIENT2.pdf}{Author-hosted full text}.

\bibitem{Sudakov69}
V. N. Sudakov, \emph{Gauss and Cauchy measures and $\varepsilon$-entropy},
Doklady Akademii Nauk SSSR \textbf{185} (1969), no.~1, 51--53 (Russian).
English translation: Soviet Mathematics Doklady \textbf{10} (1969), 310--313.

\bibitem{Sudakov}
V. N. Sudakov, \emph{Gaussian random processes and measures of solid angles
in Hilbert space}, Doklady Akademii Nauk SSSR \textbf{197} (1971), no.~1,
43--45 (Russian).
\href{https://www.mathnet.ru/eng/dan36028}{Journal record and full text}.

\bibitem{Talagrand93}
M. Talagrand, \emph{Regularity of infinitely divisible processes},
Annals of Probability \textbf{21} (1993), no.~1, 362--432.
\href{https://doi.org/10.1214/aop/1176989409}{doi:10.1214/aop/1176989409}.

\end{thebibliography}
\end{document}